\documentclass{article}

\usepackage{amsfonts,amsmath,amssymb}
\usepackage{authblk}

\newtheorem{satz}{Theorem}

\newtheorem{theorem}[satz]{Theorem}
\newtheorem{lemma}[satz]{Lemma}

\newtheorem{corollary}[satz]{Corollary}
\newtheorem{remark}[satz]{Remark}

\newtheorem{claim}[satz]{Claim}

\newcommand{\qed}{{} \hfill \mbox{$\Box$}}

\def\({\big (}
\def\){\big )}

\def\_phi{\varphi}

\title{On the Thue-Vinogradov Lemma}

\author{Jozsef Solymosi}
\affil{University of British Columbia, Vancouver\\ 
email: {solymosi@math.ubc.ca}}

\begin{document}
\maketitle
\begin{abstract}
We prove an extension of the Thue-Vinogradov Lemma. This paper is another example for the application of the polynomial method; R\'enyi polynomials and Stepanov's technique.
\end{abstract}

\section{Introduction}
In the introduction we state two classical results from elementary number theory, two lemmas from Thue and Vinogradov. In the second part of the paper we extend their results and illustrate the use of the new method by an application. 

\medskip
The lemmas of Thue and Vinogradov are clever applications of Dirichlet's box principle (also called as the pigeonhole principle). Our first result will go beyond that, it works with smaller sets.
The technique we are using here is a variant of the so called polynomial method in additive combinatorics. We are going to use R\'edei polynomials \cite{Re}, and the last step in the proof of Theorem \ref{Main} (and in its later variants) is based on Stepanov's method \cite{St}; if a degree $d$
polynomial is vanishing on a set of size $n$ with multiplicity at least $m$ then $n\leq d/m.$ The same method will be used in the last section, where we prove an inequality in additive combinatorics.

\subsection{The lemmas of Thue and Vinogradov}
Thue's Lemma is a useful tool in elementary number theory. The most famous application of the lemma is 
 to prove Fermat's theorem on sums of two squares. There is a nice description of Thue's argument in the book {\em "Proof from THE BOOK"} \cite{AZ}.
 The lemma is used in finding solutions of Diophantine equations involving quadratic forms. 
 There are various examples for such theorems and exercises in Nagell's Introduction to Number Theory \cite{Na}, and in Vinogradov's Elements of Number Theory \cite{Vi_3}.

\begin{lemma}[Thue's Lemma]\cite{Th}
Let $p$ be a prime. For any $a\in \mathbb{N},$ $p\nmid a,$ there are $x,y$
\[
x,y\in \left\{1,2,\ldots ,\left\lceil\sqrt{p}\right\rceil\right\}
\]
such that 

\begin{align*}
  ax &\equiv \pm y \pmod{p}.
\end{align*}

\end{lemma}

Thue's Lemma was extended by Vinogradov to an asymmetric form. He used it in the paper {\em "On a general theorem concerning the distribution of the residues and
non-residues of powers"} \cite[Lemma 1]{Vi_1}, where he gave an elementary proof of the P\'olya-Vinogradov inequality. His extension, the following lemma,
can be also used to find solutions for some quadratic forms, more efficiently than Thue's Lemma.

\begin{lemma}[Vinogradov's Lemma]
Let $p$ be a prime. For any $a\in \mathbb{N},$ $p\nmid a,$ and $\alpha\in \mathbb{F}_p^*,$ there are $x,y$
\[
x\in \left\{1,2,\ldots ,\alpha\right\},\quad y\in \left\{1,2,\ldots ,\left\lfloor\frac{p}{\alpha}\right\rfloor\right\}
\]
such that 

\begin{align*}
  ax &\equiv \pm y \pmod{p},
\end{align*}
or equivalently 
\begin{align*}
  a &\equiv \pm \frac{y}{x} \pmod{p}.
\end{align*}

\end{lemma}

Vinogradov's result was generalized to multiple congruences by Brauer and Reynolds in \cite{BR} where they provide a complete historic review of re-discoveries 
and generalizations of the Thue-Vinogradov lemma, up to 1951. In the same paper they proved the following result \cite[Theorem 4]{BR}.

\begin{theorem}\label{cong}
Let $g$ and $k$ be positive integers where $k$ is even, $p$ an odd prime
with $p \equiv 1 \pmod{k}$ such that $g\leq p.$ We set $h = \lceil p/g\rceil.$ If $D$ is a $k$-th power
residue, then at least one of the numbers $1^k, 2^k,\ldots , h^k$ is congruent to one of the
numbers $D, 2^kD, \ldots , (g - 1)^kD.$
\end{theorem}

Theorem \ref{cong} was also proved, independently, by Porcelly and Pall using Farey sequences in \cite{PP}.
We are going to prove an improvement on this theorem in Section \ref{Pairs}.

\section{The Extension}

The Thue-Vinogradov lemma is about initial segments providing solutions to $ax  \equiv \pm y \pmod{p}$ for all $a.$ What can we say about shorter segments? We 
are going to use the polynomial method -- in this case the R\'edei polynomial -- to prove that initial segments of $\mathbb{F}_p$ give many solutions to the above congruence. 
R\'edei polynomials were used in number theory, group theory, and in the geometry of finite fields. There is a nice survey on basic theorems and examples to such 
applications of the R\'edei polynomial (and other algebraic methods in combinatorics) in \cite{Al}.

\begin{theorem}\label{Main}
Let $p$ be a prime. For any $\alpha,\beta\in \mathbb{N},$ $\alpha(\beta+1)\leq p-1,$ there are at least
$\alpha(\beta+1)$ distinct
$a\in\mathbb{F}_p^*$ for which there are $x,y$
\[
x\in I_\alpha=\left\{1,2,\ldots ,\alpha\right\},\quad y\in I_\beta=\left\{1,2,\ldots ,\beta\right\}
\]
such that 

\begin{equation}\label{mod}
  ax \equiv \pm y \pmod{p}.
\end{equation}

\end{theorem}

\noindent
In Vinogradov's Lemma if $\alpha(\beta +1)>p,$ then the conclusion of the theorem holds for every $a\in\mathbb{F}_p^*,$ even with $y\in \left\{1,2,\ldots ,\beta-1\right\},$ so there are infinitely many cases when Vinogradov's Lemma gives a better bound (by one) if one needs to capture every $a\in\mathbb{F}_p^*.$ The importance of Theorem \ref{Main} is that it covers the range when $\alpha\beta<p,$ when simple pigeonhole arguments won't work.

\medskip
\begin{proof}
Denote $D\subset\mathbb{F}_p^*$ the set of elements $a$ which are not expressible as in (\ref{mod}).   The key of the argument is the construction of a polynomial following R\'edei \cite{Re} and Sz\H{o}nyi \cite{Szo}. Their method was specialized to Cartesian products in \cite{BSW}, in a way that we are going to follow here. The polynomial is defined as 

\[
H(x,y)=\prod_{i=0}^{{\beta}}\left(x-i\right)\prod\limits_{\scriptstyle 1\leq k\leq \alpha \atop\scriptstyle 0\leq j \leq \beta\hfill}(x+ky-j)=\prod\limits_{\scriptstyle 0\leq k\leq \alpha \atop\scriptstyle 0\leq j \leq \beta\hfill}(x+ky-j)
\]

 The important feature of the polynomial above is that whenever $b\in D,$ all roots of $H(x,b)$ are distinct elements of $\mathbb{F}_p,$
i.e. $H(x,b)$ divides $x^p-x.$ To see that, let us consider the two possible cases of repeated roots below

\renewcommand{\theenumiii}{\Roman{enumii}}

\begin{enumerate}
    \item If the second product term (with $y$-s) had two equal roots then we had 
    \begin{equation*}
  -kb+j \equiv  -k'b+j' \pmod{p},
\end{equation*}
for some $1\leq k,k'\leq\alpha$ and $0\leq j,j'\leq\beta .$ If $k=k'$ then $j=j',$ but then the two linear terms are the same which is not possible.
Note that $b\neq 0$ so
\begin{equation*}
 |k-k'|b \equiv \pm (j'-j) \pmod{p},
\end{equation*}
contradicting to the assumption $b\in D.$

    \item The remaining case is when 
\begin{equation*}
  -kb+j \equiv  j' \pmod{p},
\end{equation*}
for some $1\leq k\leq\alpha$ and $0\leq j,j'\leq\beta, $  leading to   
\begin{equation*}
  kb \equiv \pm (j'-j) \pmod{p},
\end{equation*}
contradicting to the assumption $b\in D.$

\end{enumerate}

The degree of $H$ is $\delta=\alpha\beta+\alpha+\beta+1.$ In particular, when $\alpha=\beta$ then the degree is $(\alpha+1)^2.$ It was Sz\H{o}nyi's observation in \cite{Szo} (see also in \cite{Szo2}) that there is an auxiliary polynomial of degree $p-\delta,$ denoted by $f(x,y),$ such that

\begin{equation}\label{lac}
F(x,b)=f(x,b)H(x,b)=x^p-x \quad {\text  if   } \quad b\in D .
\end{equation}

For the details on how to find $f,$ we refer to \cite{Szo} and \cite{BSW}. Let us consider $F(x,y)$ as a polynomial of $x$ with coefficients $h_i(y)\in \mathbb{F}_p[y].$

\[
F(x,y)=f(x,y)H(x,y)=F_y(x)=x^p+h_1(y)x^{p-1}+h_2(y)x^{p-2}+\ldots +h_p(y)
\]

where the degree of $h_i$ is at most $i.$ From (\ref{lac}) one can see that $h_i(y)$-s are zero for many $y$ values, whenever $y\in D.$ If $h_i(y)=0$ for more than $i$ distinct $y$ values then $h_i(y)\equiv 0.$ This is the crucial point of the application of R\'edei's method. If one can show that $h_i\not\equiv 0$ for some $i,$ then $|D|\leq i.$ When $|D|$ is small, one could use R\'edei's theorem, which describes the structure of fully reducible lacunary polynomials (like in \cite{Szo}), however we follow a simpler calculation which gives a better bound in this case. Let us check the polynomial $F(x,y)$ when $y=0.$

\begin{equation}\label{prod}
\begin{split}
    F(x,0) & =f(x,0)\left(\prod_{i=0}^{\beta}(x-i)\right)^{\alpha+1}\\
    & =x^p+c_1x^{p-1}+c_2x^{p-2}+\ldots +c_p.
\end{split}    
\end{equation}

We need to show that a polynomial with form like in (\ref{prod}) has a nonzero $c_i$ coefficient for some, not too large $i.$ Let $c_i$ denote the nonzero coefficient with the smallest index $i.$ Checking the derivatives based on the first and second rows, we see that $F'(x,0)$ will vanish with multiplicity at least $\alpha$ on at least $\beta+1$ places and it has degree $p-i-1.$ This implies that 
$p-i-1\geq \alpha(\beta+1)$ and then $|D|\leq i\leq p-1-\alpha(\beta+1)$ as needed.  \qed
\end{proof}

\begin{remark}\label{rem1}
Theorem \ref{Main} was stated for initial segments, but the same proof works if one requires 
\[
x\in \mu I_\alpha =\left\{\mu,2\mu,\ldots ,\alpha\mu\right\},\quad y\in \nu I_\beta=\left\{\nu,2\nu,\ldots ,\beta\nu\right\}
\]
for some $\nu,\mu\in\mathbb{N}$ values, where $p\nmid \nu\mu $.
\end{remark}

\begin{remark}\label{rem_uj}
It was noted by the anonymous referee and other readers of an earlier version of this paper that Theorem \ref{Main} can be improved for shorter initial segments. For example if
\[
x,y\in  I_\alpha =\left\{1,2,\ldots ,\alpha\right\},
\]
and $2\alpha^2<p,$ then the number of distinct $a\in \mathbb{F}_p^*,$ such that $a\equiv \pm x/y \pmod{p}$ is twice the number of (ordered) pairs $(u,v)\in {\mathbb{N}}^2,$ where $(u,v)=1,$ and $u,v\leq \alpha,$ which is asymptotically $\frac{12}{\pi^2}\alpha^2\sim 1.21 \alpha^2$  (See e.g. Exercise 21 b, Chapter II in \cite{Vi_3}). 

\end{remark}

Let us denote the difference set of $A\subset\mathbb{F}_p$ by $\Bar{A},$

\[
\Bar{A}=A-A=\{a-b | a,b\in A\}.
\]

Using the above notation we can state a more general theorem with slightly weaker bounds. It is practically the same as Theorem 1 in \cite{BSW}, we include it here for completeness.

\begin{theorem}\label{Main2}
Let $p$ be a prime. For any $A,B\subset \mathbb{F}_p,$ where $|A|=\alpha,|B|=\beta,$ there are at least
\[
\min(p,(\alpha-1)\beta+1)
\]
$a\in\mathbb{F}_p$ for which there are 
$
x\in \Bar{A}\setminus \{0\}, y\in \Bar{B}
$ such that 
$  ax \equiv y \pmod{p}.
$
\end{theorem}

Note that since $\Bar{A}$ and $\Bar{B}$ are symmetric about 0, we don't need the $\pm$ sign in the modular equation. The proof, which we are going to sketch below follows the proof of Theorem \ref{Main}.

\medskip
\begin{proof}
For $a=0$ the trivial solution, $  ax \equiv b-b \pmod{p},$ works with any $x\in \Bar{A}, b\in B.$  Let us denote $D\subset\mathbb{F}_p^*$ the set of elements $a$ which are not expressible as $  ax \equiv  y \pmod{p}.$  
The R\'edei polynomial is now defined as 

\begin{equation}\label{Red+}
H(x,y)=\prod\limits_{\scriptstyle 1\leq k\leq \alpha \atop\scriptstyle 1\leq j \leq \beta\hfill}(x+a_ky-b_j).
\end{equation}

Whenever $d\in D,$ all roots of $H(x,d)$ are distinct elements of $\mathbb{F}_p,$
i.e. $H(x,d)$ divides $x^p-x.$ If we had $(x+a_kd-b_j)=(x+a_\ell d-b_s)$ then $  (a_k-a_\ell) d \equiv b_j-b_s \pmod{p},$ contradicting the selection $d\in D.$ 
The degree of $H$ is $\delta=\alpha\beta.$ There is an auxiliary polynomial of degree $p-\delta,$ denoted by $f(x,y),$ such that

\begin{equation}\label{lac2}
F(x,d)=f(x,d)H(x,d)=x^p-x \quad {\text  if   } \quad d\in D .
\end{equation}

Let us consider $F(x,y)$ as a polynomial of $x$ with coefficients $h_i(y)\in \mathbb{F}_p[y].$

\[
F(x,y)=f(x,y)H(x,y)=F_y(x)=x^p+h_1(y)x^{p-1}+h_2(y)x^{p-2}+\ldots +h_p(y)
\]

where the degree of $h_i$ is at most $i.$ If we show that $h_i\not\equiv 0$ for some $i,$ then $|B|\leq i.$  The polynomial  when $y=0$ is

\begin{equation}\label{prod2}
\begin{split}
    F(x,0) & =f(x,0)\left(\prod_{i=1}^{\beta}(x-b_i)\right)^{\alpha}\\
    & =x^p+c_1x^{p-1}+c_2x^{p-2}+\ldots +c_p.
\end{split}    
\end{equation}

Let $c_i$ denote the nonzero coefficient with the smallest index $i.$ Checking the derivatives based on the first and second rows, we see that $F'(x,0)$ will vanish with multiplicity at least $\alpha-1$ on at least $\beta$ places and it has degree $p-i-1.$ This implies that 
$p-i-1\geq (\alpha-1)\beta$ and then $|D|\leq i\leq p-1-(\alpha-1)\beta$ as needed.  \qed
\end{proof}

\medskip

Let $d>1$ be a divisor of $p-1$ and let $Z_d$ be a multiplicative subgroup of size $d$ inside $GF(p)$. If there is an $A\subset \mathbb{F}_p$ such that $\Bar{A}\subset \{Z_d\cup 0\}$ then by  applying Theorem \ref{Main2} with $A=B$ we obtain the following result, which was recently proved by Hanson and Petridis \cite{HP}. (See also Theorem 1. in \cite{BSW})

\begin{corollary}\label{Cor}
Let $A\subset \mathbb{F}_p$ be a set such that $A-A\subset Z_d\cup\{0\}$. Then
\[ |A|(|A|-1)\leq d.\]
\end{corollary}

\medskip
 A slightly stronger statement in Theorem \ref{Main2} holds when $0\not\in A.$

\begin{theorem}
Let $A\subset \mathbb{F}_p^*,$ $B\subset \mathbb{F}_p,$ where $|A|=\alpha,|B|=\beta.$ There are at least
\[
\min(p,\alpha\beta+1)
\]
$a\in\mathbb{F}_p$ for which there are 
$
x\in \left\{ \{A\cup\Bar{A}\}\setminus \{0\}\right\},$ and $ y\in \Bar{B}$ 
 such that 
$  ax \equiv y \pmod{p}.
$
\end{theorem}
Indeed, in this case instead of polynomial (\ref{Red+}) we can use 

\begin{equation*}
H(x,y)=\prod_{\ell=1}^\beta (x-b_j)\prod\limits_{\scriptstyle 1\leq k\leq \alpha \atop\scriptstyle 1\leq j \leq \beta\hfill}(x+a_ky-b_j),
\end{equation*}

increasing the degree of $H(x,y)$ by $\beta.$ The roots are still distinct for any $d\in D,$ since $-b_\ell=a_id-b_j$ would lead to the $  ad \equiv y \pmod{p}$ equation where $x\in A$ and $y\in \Bar{B}.$ The polynomial  when $y=0$ now is

\begin{equation*}
    F(x,0)  =f(x,0)\left(\prod_{i=1}^{\beta}(x-b_i)\right)^{\alpha+1}
\end{equation*}
with the $\alpha+1$ exponent instead of $\alpha,$ leading to the improvement.

\section{Congruent pairs}\label{Pairs}
In this section we illustrate how to use Theorem \ref{Main} when we need many, almost $p$ solutions in (\ref{mod}). The proof is similar to classical applications of the Thue-Vinogradov inequality. We are going to show a variant of Theorem \ref{cong} stated in the introduction. 

\begin{theorem}\label{pair_the}
Let $g$ and $k$ be positive integers where $k$ is even, $p$ an odd prime
with $p \equiv 1 \pmod{k}$ such that $g\leq p.$ Let $h\in\mathbb{N}$ be a number given by \[h= \left\lceil\frac{p-k-g}{g-1}\right\rceil.\] 
If $D$ is a $k$-th power
residue, then at least one of the numbers $1, 2^k,\ldots , h^k$ is congruent to one of the
numbers $D, 2^kD, \ldots , (g - 1)^kD.$
\end{theorem}

If $g\geq h$ then the above $h$ is at most as as large as in Theorem \ref{cong} and $h$ is smaller here by at least one whenever $g(k+g)\geq p.$

\medskip
\begin{proof}
The equation $x^k \equiv  D \pmod{p}$ has $k$ solutions (see e.g. in \cite{Vi_3}, page 113).
By Theorem \ref{Main} if 

\[(g-1)(h+1)+1\geq p-k,\] 

which is provided by the condition
\[h= \left\lceil\frac{p-k-g}{g-1}\right\rceil,\] 

then there is an $a\in\mathbb{F}_p$ such that $a^k \equiv  D \pmod{p}$ and

\begin{equation}
  ax \equiv \pm y \pmod{p}.
\end{equation}

where $x,y$ are
\[
x\in \left\{1,2,\ldots ,g-1\right\},\quad y\in \left\{1,2,\ldots ,h\right\}.
\]

The following equations 

\begin{equation*}
\begin{split}
    ax \equiv \pm y \pmod{p}\\
    a^kx^k \equiv  y^k \pmod{p}\\
    Dx^k \equiv  y^k \pmod{p}
\end{split}    
\end{equation*}

show that there is at least one congruent pair between 
\[
\left\{D, 2^kD, \ldots , (g - 1)^kD\right\}\quad \text{and}\quad\left\{1, 2^k,\ldots , h^k\right\} ,
\]
as required. \qed
\end{proof}

\section{Sumsets vs. Directions}
In this section we are going to leave the Cartesian product structure and prove a result which generalizes Theorem \ref{Main2} and other results.
One of the most striking applications of R\'edei's method is the bound on the number of directions determined by a set of points in the affine  plane over the
finite field $GF(q)$ of $q$ elements. Given a set $M$ of $n$ points what is the minimum number
of directions determined by $M$? We say that the direction $m$ is
determined by $M$ if there is a line $mx+b-y=0$ spanned by two points of
$M,$ i.e. there are points $(a_i,b_i),(a_j,b_j)\in M$ such that $m=(a_i-a_j)/(b_i-b_j)$ if $b_i\neq b_j.$ If $b_i=b_j$ and $a_i\neq a_j$ then the two points determine the $m=\infty$ direction. 

In Theorem \ref{Main2} we proved a lower bound on the number of directions determined by a Cartesian product. It was better than Sz\H{o}nyi's bound in \cite{Szo,Szo2}, due to the special structure of the pointset. In the next result we generalize Theorem \ref{Main2}. 

Given an $n$-element subset $S\subset\mathbb{F}_p^2,$ and an $\alpha\in\mathbb{F}_p^*.$ Let us suppose that $n<p.$ We define the weighted sumset
\[
\Delta_\alpha=\{\alpha a_i+b_i\quad | \quad (a_i,b_i)\in S\},
\]
and the ratio set

\[
Q=\left\{\frac{a_i-a_j}{b_i-b_j}\quad \large| \quad (a_i,b_i),(a_j,b_j)\in S, b_i\neq b_j\right\}.
\]

The ratio set contains all directions determined by $S$ with the possible exception of the $(\infty)$ direction.

 \begin{theorem}\label{dir1}
 With the above notation, 
 if $S$ is not collinear, i.e. there are no elements $m,\beta\in \mathbb{F}_p$ such that $ma_i+\beta-b_i\equiv 0 \pmod{p}$ for all $(a_i,b_i)\in S,$ then 
 $|Q|\geq|S|-|\Delta_\alpha|+1.$
 \end{theorem}
 
\begin{proof}
We are going to use the R\'edei polynomial as before. Set

\begin{equation}\label{Dir+}
H(x,y)=\prod\limits_{(a_i,b_i)\in S}(x+a_iy-b_i),
\end{equation}
 and find $f(x,y)$ such that $f(x,y_0)H(x,y_0)=x^p-x$ whenever $y_0\not\in Q.$ Let us check the polynomial when we set $y=-\alpha.$
 
 \begin{equation}\label{prod3}
\begin{split}
    F(x,\alpha) & =f(x,\alpha)\prod\limits_{(a_i,b_i)\in S}(x-\alpha a_i-b_i)\\
    & =x^p+c_1x^{p-1}+c_2x^{p-2}+\ldots +c_p.
\end{split}    
\end{equation}

 Like in the proof of Theorem \ref{Main}, we check the derivatives to show that there is a small index $i$ where $c_i\neq 0,$ so $Q$ is large. A root $\alpha a_i+b_i$ is a multiple root if there is an $(a_j,b_j)\in S,$ $i\neq j,$ such that $\alpha a_i+b_i\equiv \alpha a_j+b_j \pmod{p}.$ The derivative of the polynomial in (\ref{prod3}) has at least $d=|S|-|\Delta_\alpha|$ roots, so $i-1\leq p-d,$ unless $F(x,\alpha)=(x+c)^p,$ when $S$ is collinear.
 \qed
 \end{proof}
 
 \medskip
 Note that setting $\alpha=0$ for a Cartesian product, $S,$ gives back Theorem \ref{Main2}.
 
 \section{Acknowledgements}
I would like to thank the anonymous referee for the helpful report and in particular for the improvement mentioned in Remark \ref{rem_uj}. I am also thanful to  Andrew Granville, Ilya Shkredov, and Ethan White for helpful discussions. Research was supported in part by an NSERC Discovery grant, OTKA K 119528 and NKFI KKP 133819 grants.

\end{document}